\documentclass[12pt]{amsart}
\usepackage[active]{srcltx}
\usepackage[centertags]{amsmath}
\usepackage{latexsym}
\usepackage{amssymb}

\DeclareMathOperator{\charac}{char}

\DeclareMathOperator{\SRIM}{\mathit{SRIM}}

\newcommand{\C}{\mathbb C}
\newcommand{\F}{\mathbb F}

\newcommand{\N}{\mathbb N}

\newcommand{\Z}{\mathbb Z}
\newcommand{\PGL}{\mathrm{PGL}}

\newcommand{\hide}[1]{}
\newtheorem{dummy}{Dummy}

\newtheorem{lemma}[dummy]{Lemma}

\newtheorem{theorem}[dummy]{Theorem}

\newtheorem{cor}[dummy]{Corollary}

\theoremstyle{definition}

\theoremstyle{remark}

\begin{document}

\bibliographystyle{amsalpha}
\author{Sandro Mattarei}
\email{smattarei@lincoln.ac.uk}
\address{School of Mathematics and Physics\\
  University of Lincoln\\
  Brayford Pool\\
  Lincoln, LN6 7TS\\
  United Kingdom}
\author{Marco Pizzato}
\email{marco.pizzato1@gmail.com}
\title{Generalizations of self-reciprocal polynomials}
\begin{abstract}
A formula for the number of monic irreducible self-reciprocal polynomials, of a given degree over a finite field, was given by Carlitz in 1967.
In 2011 Ahmadi showed that Carlitz's formula extends, essentially without change,
to a count of irreducible polynomials arising through an arbitrary
{\em quadratic transformation}.
In the present paper we provide an explanation for this extension, and a simpler proof of Ahmadi's result,
by a reduction to the known special case of self-reciprocal polynomials and a minor variation.
We also prove further results on polynomials arising through a quadratic transformation, and through some special transformations of higher degree.
\end{abstract}

\maketitle

\section{Introduction}\label{sec:intro}

A polynomial $f(x)$, of positive degree, over a field, is said to be {\em self-reciprocal} if
$x^{\deg f}f(1/x)=f(x)$.
Every monic irreducible self-reciprocal polynomial except for $x+1$ has even degree.
The abbreviation {\em srim} is in use for {\em self-reciprocal irreducible monic}.
Various counting formulas for irreducible polynomials of certain types exist,
on the model of Gauss's formula
$(1/n)\sum_{d\mid n}\mu(d)q^{n/d}$
for the total number of monic irreducible polynomials of degree $n$ over the field $\F_q$ of $q$ elements.
Carlitz proved in~\cite{Carlitz:srim} that the number $\SRIM(2n,q)$ of {\em srim} polynomials of degree $2n$ over a finite field $\F_q$
is given by
\begin{equation}\label{eq:Carlitz}
\SRIM(2n,q)=
\begin{cases}
\displaystyle
\frac{q^n-1}{2n}
 &\text{if $q$ is odd and $n$ is a power of $2$,}
\\
\displaystyle
\frac{1}{2n}\sum_{d\mid n,\ \text{$d$ odd}}
\mu(d)q^{n/d}
 &\text{otherwise.}
\end{cases}
\end{equation}

Simpler proofs of Equation~\eqref{eq:Carlitz} were given by Cohen~\cite{Cohen:irreducible} and Meyn~\cite{Meyn}.
The latter proof applies M\"obius inversion to the fact, of which our Theorem~\ref{thm:Meyn} below is a slight generalization, that
the nonlinear {\em srim} are exactly the nonlinear irreducible factors of polynomials of the form $x^{q^{n}+1}-1$.
The proofs of Carlitz and Cohen rely, in a crucial way, on
the well-known fact that any self-reciprocal polynomial of degree $2n$ over a
field can be expressed as $x^n\cdot f(x+x^{-1})$ for some polynomial $f(x)$ of degree $n$.
(This fact also featured in Meyn's paper, but was used only for further developments.)
This point of view motivated Ahmadi~\cite{Ahmadi:Carlitz} to study polynomials
obtained from such $f(x)$ through a more general {\em quadratic transformation,} namely, polynomials of the form
$h(x)^n\cdot f\bigl(g(x)/h(x)\bigr)$,
where $g(x)$ and $h(x)$ are coprime polynomials with $\max(\deg g,\deg h)=2$,
and $n=\deg f$ as above.
The special case of {\em srim} polynomials arises when $g(x)/h(x)=(x^2+1)/x=x+x^{-1}$.
Ahmadi found that the number of such polynomials which are irreducible of degree $2n>2$ over $\F_q$,
for a given $g(x)/h(x)$, equals $SRIM(2n,q)$ except, for $q$ even, when both $g(x)$ and $h(x)$
miss the linear term, in which case no irreducible polynomials arise.

One goal of this paper is to give a simple explanation
of Ahmadi's conclusion that, aside from that exceptional case, Carlitz's count of {\em srim} polynomials
extends unchanged to a count of the polynomials obtained through an arbitrary fixed quadratic transformation.
The reason is that the quadratic rational expression $g(x)/h(x)$ employed may be composed with
linear fractional expressions $(ax+b)/(cx+d)$ on both sides without changing the resulting count of irreducible polynomials.
By doing so $g(x)/h(x)$ can be brought to one of only two forms over $\F_q$, which in the odd characteristic case
are $x+x^{-1}$ and $x+\sigma x^{-1}$, where $\sigma$ is any fixed non-square in $\F_q^{\ast}$.
We explain one way to perform this reduction in Section~\ref{sec:quadratic}.
At this point, half the cases follow from Carlitz's result, and the other half from a straightforward variation.
This produces a shorter and simpler proof of Ahmadi's result,
which we present in Section~\ref{sec:counting}.

In the rest of the paper we present supplementary results on this topic.
Meyn's proof in~\cite{Meyn} of Carlitz's counting formula for {\em srim} polynomials was based
on viewing them as irreducible factors of polynomials of the form $x^{q^n+1}-1$.
That explicit description of all irreducible factors of $x^{q^n+1}-1$
as self-reciprocal polynomials of certain degrees (plus $x-1$ when $q$ is odd),
admits a much more general version which we present in Section~\ref{sec:explicit}.
In theorem~\ref{thm:Meyn_generalized} there, irreducible polynomials arising through an arbitrary quadratic transformation
are used to describe the complete factorization of a certain related polynomial of the form
$ax^{q^n+1}-b(x^{q^n}+x)+c$.
Note that factorizing polynomials of this form is also the subject of~\cite{StiTop},
but as we discuss at the end of Section~\ref{sec:explicit}
there is little overlap with our results as the goals are different.

We have mentioned how the well-known characterization of even-degree self-reciprocal polynomials as those of the form $x^{\deg f}\cdot f(x+1/x)$
was a simple but essential fact for various investigations of self-reciprocal polynomials.
This is also the case in the present paper, with the definition of being self-reciprocal as, appropriately formulated, invariance under the substitution $x\mapsto 1/x$
first slightly generalized to invariance under $x\mapsto\sigma/x$ in Lemma~\ref{lemma:invariant_quadratic}, and then
to invariance under any involutory M\"obius transformation in Lemma~\ref{lemma:invariant_generalized}.
We devote Section~\ref{sec:variations} to a discussion of alternate proofs of these results, and to variations
concerning invariance under M\"obius transformations of higher order.

The research leading to this paper begun when the second author was a PhD student at the University of Trento,
Italy, under the supervision of the first author.
Part of these results have appeared among other results in~\cite{Pizzato:thesis}.

\section{Quadratic transformations}\label{sec:quadratic}

Let $K$ be any field and fix a {\em quadratic rational expression} $R(x)=g(x)/h(x)$,
where $g(x),h(x)\in K[x]$ are coprime polynomials with $\max(\deg g,\deg h)=2$.
This induces a {\em quadratic transformation} of polynomials in $K[x]$,
which sends (zero to zero if we like, and) a nonzero polynomial $f(x)$ to the polynomial
$f_R(x):=h(x)^{\deg f}\cdot f\bigl(g(x)/h(x)\bigr)$.
Thus, the quadratic transformation is given by the substitution $x\mapsto R(x)$ into $f(x)$
(or applying {\em pre-composition} with $x\mapsto R(x)$ if we prefer), followed with multiplication
by the least power of $h(x)$ required to clear denominators and ensure that $f_R(x)$ is actually a polynomial.

A formal treatment of a general quadratic transformation, associated to a quadratic rational expression $R(x)$,
is encumbered by some technicalities.
A harmless one is a scalar factor ambiguity in $f_R(x)$ upon writing
$R(x)=g(x)/h(x)$ in an equivalent form $(ag(x))/(ah(x))$.
One may resolve this by including a normalization to the unique monic scalar multiple in the definition of $f_R(x)$,
but we rather not do so as it may create other issues.

More disturbing is the fact that in some cases a quadratic transformation may not double the degree of a polynomial, as seen, for example, in
$x^n\mapsto (x^2)^n\cdot(1/x^2)^n=1$ when $g(x)=1$ and $h(x)=x^2$.
More generally, $\deg f_R=2\deg f$ unless, in self-explanatory projective language, $(g/h)(\infty)$ is a root of $f$.
Written out explicitly, that occurs exactly when $h_2\neq 0$ and $f(g_2/h_2)=0$,
and hence cannot occur for $f$ irreducible with $\deg(f)>1$.
We generally work on this assumption, which was also made in~\cite{Ahmadi:Carlitz}.
However, we will consider polynomials $f$ of degree one in the last part of Section~\ref{sec:counting} (after the proof of Theorem~\ref{thm:Ahmadi}),
and allow possibly reducible polynomials $f$ in Section~\ref{sec:explicit}.
In both instances we will explain how to deal with the resulting issue of a possible drop in degree.
A related issue is that, when $\deg f_R<2\deg f$ only, the transformed polynomial $f_R$ may be irreducible without $f$ being irreducible:
with $R(x)=1/x^2$ as in the previous example, the transformation takes the reducible polynomial $xf(x)$ to $f_R(x)$, which may be irreducible.

The key to our proof of Ahmadi's result in~\cite{Ahmadi:Carlitz} is that any quadratic rational expression $R(x)=g(x)/h(x)$ can be brought to a simple special form
by {\em pre-} and {\em post-composition} with certain invertible transformations of the form
$x\mapsto(ax+b)/(cx+d)$, that is, elements of the {\em M\"obius group}.
Recall that the M\"obius group, over a field $K$, is isomorphic with the projective general linear group $\PGL(2,K)$,
with the image of the matrix
$\left[\begin{smallmatrix}
a&b\\c&d
\end{smallmatrix}\right]$
in $\PGL(2,K)$ corresponding to the {\em M\"obius transformation} of the previous sentence.
Because the M\"obius group is generated by
the affine maps $x\mapsto ax+b$ (with $a\in K^\ast$ and $b\in K$), and the inversion map $x\mapsto 1/x$,
reducing $R(x)$ to a special form can be done by repeated and appropriate use of only those maps, as we show
in Theorem~\ref{prop:canonical-deg2} below.

Before doing that we show that our goal of counting the irreducible polynomials of the form
$f_R(x)=h(x)^{\deg f}\cdot f\bigl(g(x)/h(x)\bigr)$ is (essentially) not affected by composing
$g(x)/h(x)$, on either side, with maps of those two types.

\begin{lemma}\label{lemma:unaffected}
Let $R(x)=g(x)/h(x)$ be a quadratic rational expression over $\F_q$, and fix $n>1$.
Then the number of irreducible polynomials in $\F_q[x]$
of the form $f_R(x)=h(x)^{\deg f}\cdot f\bigl(g(x)/h(x)\bigr)$ for some irreducible $f\in\F_q[x]$ of degree $n$,
does not change upon composing the quadratic expression
$g(x)/h(x)$, on either side, with affine maps or the inversion map.
\end{lemma}

\begin{proof}
Note that $\deg f_R=2\deg f=2n$ because of our assumption that $f$ is irreducible of degree $n>1$.

Pre-composition with (invertible) affine maps
clearly does not affect the irreducibility of
$f_R(x)=h(x)^{\deg f}\cdot f\bigl(g(x)/h(x)\bigr)$.

Pre-composing $g(x)/h(x)$ with the inversion map before applying the quadratic transformation to $f$ results in
$(x^2h(1/x))^{\deg f}\cdot f\bigl(g(1/x)/h(1/x)\bigr)$.
This coincides with the {\em reciprocal} polynomial $x^{\deg f_R}f_R(1/x)$ of $f_R(x)$
precisely because $\deg f_R=2\deg f$.

Post-compositions do not generally preserve irreducibility of $f_R(x)$.
However, if $\tilde R(x)=ag(x)/h(x)+b$
then the map $f(x)\mapsto\tilde f(x)=f(ax+b)$
is a degree-preserving bijection from the set of irreducible polynomials $f(x)$ such that
$f_{\tilde R}(x)$ is irreducible, onto the set of irreducible polynomials $\tilde f$ such that
$\tilde f_R(x)$ is irreducible, because
$f_{\tilde R}(x)=\tilde f_R(x)$.

Similarly, if $\tilde R(x)=h(x)/g(x)$
then the map $f(x)\mapsto\tilde f(x)=x^{\deg f}f(1/x)$
is a degree-preserving bijection from the set of irreducible polynomials $f$ with $\deg f>1$ such that
$f_{\tilde R}(x)$ is irreducible, onto the set of irreducible polynomials $\tilde f$ with $\deg\tilde f>1$ such that
$\tilde f_R(x)$ is irreducible,
again because
$f_{\tilde R}(x)=\tilde f_R(x)$.
Here the assumption $\deg f>1$ serves to exclude the exceptional case $f(x)=x$.
\end{proof}

The following result and its proof show how to bring $g(x)/h(x)$ to particularly simple forms
through the transformations described in Lemma~\ref{lemma:unaffected}.
This reduction can be done over an arbitrary field $K$.

\begin{theorem}\label{prop:canonical-deg2}
Let $K$ be a field, and let $g,h$ be coprime polynomials in $K[x]$ with
$\max(\deg g,\deg h)=2$.
Then the quadratic rational expression $g(x)/h(x)$, upon composing on both
sides with affine maps $x\mapsto ax+b$, and the inversion map $x\mapsto 1/x$, repeatedly and in some order, can be brought to the form
$x+\sigma x^{-1}$ for some $\sigma\in K^{\ast}$, or, when $\charac K=2$, to the form $x^2$.
\end{theorem}

\begin{proof}
Write $g(x)=g_2x^2+g_1x+g_0$
and $h(x)=h_2x^2+h_1x+h_0$.
Most of our work will serve to remove the quadratic term from the denominator,
while leaving a linear term if that is possible.

We first deal with the rather special case where $g_2h_1=g_1h_2$ and $g_1h_0=g_0h_1$.
Because $g(x)$ and $h(x)$ are coprime these conditions imply $g_1=h_1=0$, and hence
$g(x)/h(x)=(g_2x^2+g_0)/(h_2x^2+h_0)$.
Replacing $g(x)/h(x)$ with $g(x+1)/h(x+1)$ will get us away from
this special situation, except when $K$ has characteristic two.
In that case,
if $h_2=0$ then
$(h_0/g_2)\cdot\bigl(g(x)/h(x)-g_0/h_0\bigr)=x^2$,
as desired.
If $K$ has characteristic two and $h_2\neq 0$ then
$1/\bigl(g(x)/h(x)-g_2/h_2\bigr)$
has no quadratic term at the denominator,
and proceeding as in the previous case we can reach the desired form $x^2$.

As we mentioned, if the characteristic of $K$ is not two then, possibly after substituting $x$ with $x+1$,
we may arrange for at least one of the conditions $g_2h_1\neq g_1h_2$ and $g_1h_0\neq g_0h_1$ to hold.
Possibly after replacing $g(x)/h(x)$ with $g(1/x)/h(1/x)$
we may assume that the former holds.
If $h_2=0$ then our expression has the form
$(g'_2x^2+g'_1x+g'_0)/(h'_1x+h'_0)$, with $h'_1\neq 0$.
Otherwise, $1/\bigl(g(x)/h(x)-g_2/h_2\bigr)$
will have that form.

Finally, applying the substitution $x\mapsto x-h'_0/h'_1$
and then multiplying the resulting expression by a suitable constant brings it to the form
$(x^2+g''_1x+g''_0)/x$,
and then
$(x^2+g''_1x+g''_0)/x-g''_1=x+\sigma/x$,
where
$\sigma=g''_0\in K^{\ast}$.
\end{proof}

As a distinguished example, over a field $K$ of characteristic not two the procedure described in the above proof brings $g(x)/h(x)=x^2$ to the form $x+x^{-1}$.
This can also be achieved in one go as the composition
\[
\frac{2x+2}{-x+1}\circ x^2\circ\frac{x-1}{x+1}=x+\frac{1}{x},
\]
which is essentially an application of the {\em Cayley transform.}
This equivalence of $x^2$ and $x+x^{-1}$ explains why the number of irreducible monic polynomials in $\F_q[x]$, for $q$ odd, having the form $f(x^2)$ and degree $2n$,
which can be read off~\cite[Theorem~3]{Cohen:irreducible} as a special case,
coincides with the number of srim polynomials of the same degree.

Given a quadratic rational expression $g(x)/h(x)$, we can tell which of the special forms of Theorem~\ref{prop:canonical-deg2}
it can be brought to without actually performing the full reduction procedure, but rather considering the derivatives $g'$ and $h'$ of $g$ and $h$.
In fact, both $g'$ and $h'$ vanish exactly when $\charac{K}=2$ and $g(x)/h(x)$ can be brought to the form $x^2$.
Assuming this is not the case, we know that $g(x)/h(x)$ can be brought to the form $x+\sigma x^{-1}$ for some $\sigma\in K^{\ast}$,
and we only need to find an appropriate value of $\sigma$.
Because $a(x/a+\sigma a/x)=x+\sigma a^2/x$, the value of $\sigma$ can be multiplied by any nonzero square.
Consider the polynomial $g'h-gh'$, which is at most quadratic as its quadratic term equals $(g_2h_1-g_1h_2)x^2$.
In case this quadratic term vanishes, replace $g(x)/h(x)$ with $\bigl(x^2g(1/x)\bigr)/\bigl(x^2h(1/x)\bigr)$
as in the proof of Theorem~\ref{prop:canonical-deg2}, and then the new $g'h-gh'$ will be quadratic.
Then we may take as $\sigma$ the discriminant of $g'h-gh'$.
This is because the rest of the proof only used post-compositions with inversion or affine maps,
which replace $g'h-gh'$ with a nonzero scalar multiple, and pre-composition with affine maps,
whose effect on $g'h-gh'$ does not change its discriminant (up to squares).

In conclusion, under the equivalence which is implicit in Theorem~\ref{prop:canonical-deg2},
and denoting by $(K^\ast)^2$ the set of squares in $K^\ast$,
quadratic rational expressions $g(x)/h(x)$ are naturally classified by the quotient group $K^\ast/(K^\ast)^2$ in characteristic not two,
and by $K^\ast/(K^\ast)^2$ plus one element in characteristic two, with the extra element occurring when $g(x)/h(x)\in K(x^2)$.

In particular, when $K$ is a finite field $\F_q$ and $q$ is odd,
any quadratic rational expression can be brought to precisely one of the forms
$x+x^{-1}$ and $x+\sigma_0 x^{-1}$, where $\sigma_0$ is a fixed nonsquare in $\F_q$.
When $K=\F_q$ with $q$ even,
any quadratic rational expression can be brought to precisely one of the forms $x+x^{-1}$ and $x^2$.
However, the latter form contributes no irreducible polynomials, as $f(x^2)$
is the square of a polynomial in $\F_q[x]$ if $q$ is even.

\section{Counting irreducible polynomials obtained through a quadratic transformation}\label{sec:counting}

Theorem~\ref{prop:canonical-deg2}, together with the discussion which follows it,
reduces the problem of counting the irreducible polynomials of the form $f\bigl(g(x)/h(x)\bigr)$
to the cases where the quadratic expression $g(x)/h(x)$
has the special form $x+\sigma x^{-1}$.
Thus, we see that about {\em half} the possibilities for $g(x)/h(x)$ when $q$ is odd
(those where $\sigma$ is a square in $\F_q^\ast$), and {\em all} the possibilities when $q$ is even,
have already been dealt with by Carlitz's count of self-reciprocal irreducible polynomials.
In particular, we can already conclude that, for $q$ even and
$g(x)/h(x)\not\in \F_q(x^2)$, the number of irreducible monic
polynomials of degree $2n$ in $\F_q[x]$ having the form $f\bigl(g(x)/h(x)\bigr)$
is still given by Carlitz's formula for the number of self-reciprocal irreducible monic polynomials of the same degree.

The missing half possibilities for $q$ odd, which occur when $\sigma$ is not a square in $\F_q^\ast$,
can be covered with a simple extension of any of the various proofs
for Carlitz's formula which are available, found in~\cite{Carlitz:srim,Cohen:irreducible,Meyn}.
We have chosen a presentation close to that of~\cite{Meyn}.

We start with a slight extension of the well-known fact that any self-reciprocal polynomial of degree $2n$ over a
field can be expressed as $x^n\cdot f(x+x^{-1})$ for some polynomial $f(x)$ of degree $n$.
This simple but crucial fact can be proved in many ways, and in view of generalizations we review several lines of proof in Section~\ref{sec:variations},
including a constructive proof based on Dickson polynomials.
Here we present what we feel is the most elementary proof.

\begin{lemma}\label{lemma:invariant_quadratic}
Let $K$ be a field, let $\sigma\in K^{\ast}$, and let $F\in K[x]$ be a polynomial of even degree $2n$.
Then $x^{2n}\cdot F(\sigma/x)=\sigma^nF(x)$ holds if, and only if,
$F(x)=x^n\cdot f(x+\sigma/x)$ for some $f\in K[x]$ of degree $n$.
\end{lemma}

\begin{proof}
If $f\in K[x]$ has degree $n$, then
$F(x)=x^n\cdot f(x+\sigma/x)$
is a polynomial of degree $2n$,
and clearly satisfies
$x^{2n}\cdot F(\sigma/x)=\sigma^nF(x)$.

We can prove the converse implication by a simple linear algebra argument provided we release the even integer $2n$ from being equal to $\deg(F)$, as follows.
Given a non-negative integer $n$, the assignment $f\mapsto F$, where $F(x)=x^n\cdot f(x+\sigma/x)$,
defines an injective $K$-linear map from the $(n+1)$-dimensional space of
polynomials $f\in K[x]$ of degree at most $n$,
into the space $V$ of polynomials $F\in K[x]$ having degree at most $2n$
and satisfying the condition $x^{2n}\cdot F(\sigma/x)=\sigma^nF(x)$.
Written in terms of the coefficients of $F(x)=\sum_{k=0}^{2n}b_kx^k$
the condition amounts to $b_{n-k}=b_{n+k}\sigma^k$ for $0<k\le n$.
Because these $n$ equations are linearly independent
we see that $V$ has dimension $n+1$,
and so the linear map under consideration is bijective.
Because $\deg(F)=n+\deg(f)$, if $\deg(F)=2n$ then the corresponding $f$ satisfies $\deg(f)=n$ as required.
\end{proof}

As in the special case $\sigma=1$ of self-reciprocal polynomials,
the condition $x^{2n}\cdot F(\sigma/x)=\sigma^nF(x)$
for a polynomial $F$ of degree $2n$ can be checked from knowledge of all the roots of $F$
in a splitting field with their multiplicities.
For simplicity assume $F(x)$ coprime with $x^2-\sigma$, which will be satisfied in our application below.
Then the condition $x^{2n}\cdot F(\sigma/x)=\sigma^nF(x)$ is equivalent to
$\sigma/\xi$ being a root of $F$ along with each root $\xi$ of $F$, and of the same multiplicity.
This is easily seen upon writing $F(x)=\prod_{i=1}^{2n}(x-\xi_i)$ over a splitting field, once assumed monic as we may.
The proof of a more general fact will be given in Lemma~\ref{lemma:checking}.

Now we specialize $K$ to a finite field $\F_q$.
The next result we need is the following slight generalization of~\cite[Theorem~1]{Meyn}, which was the case $\sigma=1$.

\begin{theorem}\label{thm:Meyn}
Let $\sigma\in\F_q^{\ast}$, and let $\mathcal{I}_{\sigma}$ be the set of all
monic irreducible polynomials $F\in\F_q[x]$ of even degree
which satisfy
$x^{2n}\cdot F(\sigma/x)=\sigma^nF(x)$,
where $2n=\deg F$.
Then the polynomial
\[
H(x)=\frac{x^{q^n+1}-\sigma}{(x^2-\sigma,x^{q^n-1}-1)}
\]
equals the product of all $F\in\mathcal{I}_{\sigma}$
of degree a divisor of $2n$ which does not divide $n$.
\end{theorem}

Note that the denominator in the above expression for $H(x)$ equals the greatest common divisor
$(x^2-\sigma,x^{q^n+1}-\sigma)$, and hence divides the numerator.
Also, its degree equals the number of distinct square roots of $\sigma$ in $\F_{q^n}$.
Consequently,
when $q$ is odd we have
$H(x)=(x^{q^n+1}-\sigma)/(x^2-\sigma)$
unless $n$ is odd and $\sigma$ is not a square in $\F_q$, in which case
$H(x)=x^{q^n+1}-\sigma$.
When $q$ is even we have
$H(x)=(x^{q^n+1}-\sigma)/(x-\sigma^{q/2})$.

\begin{proof}[Proof of Theorem~\ref{thm:Meyn}]
The field $\F_{q^{2n}}$ contains a splitting field for $H(x)$.
The roots of $H(x)$ are all distinct, and they are exactly all elements of $\F_{q^{2n}}$ such that
$\xi^{q^n}=\sigma/\xi\neq\xi$.
In particular, the orbit of each root of $H(x)$ under the automorphism $\alpha\mapsto\alpha^q$ of $\F_{q^{2n}}$ has
length some divisor of $2n$ which does not divide $n$.
To each orbit there corresponds a monic irreducible factor of $H(x)$ over $\F_q$, having its elements as roots.

If $F(x)$ is an irreducible factor of $H(x)$, hence of degree $2n/d$ with $d$ an odd divisor of $n$,
then for each root $\xi$ of $F$
the element $\xi^{q^n}=\sigma/\xi$ is also a root.
Because all roots of $F$ are necessarily simple,
and because $F(x)$ is coprime with $x^2-\sigma$,
we conclude that $F\in\mathcal{I}_{\sigma}$.

Conversely, if $F\in\mathcal{I}_{\sigma}$ has degree $2n/d$, with $d$ an odd divisor of $n$,
then $F$ has all its roots in $\F_{q^{2n}}$, say
$\xi,\xi^q,\ldots,\xi^{q^{2n/d-1}}$.
The defining condition of $\mathcal{I}_{\sigma}$ implies that $\sigma/\xi$ is also a root, and hence
$\sigma/\xi=\xi^{q^k}$ for some integer $k$ with $0<k<2n/d$.
But then $\xi^{q^{2k}}=(\sigma/\xi)^{q^k}=\sigma/\xi^{q^k}=\xi$, forcing $k=n/d$.
From $\xi^{q^{n/d}}=\sigma/\xi$
and $\xi^{q^{2n/d}}=\xi$
we now infer
$\xi^{q^n}=\xi^{q^{2n/d}}=\sigma/\xi$, and hence
$F(x)$ divides $x^{q^n+1}-\sigma$.
Also, $F$ cannot divide $x^2-\sigma$, otherwise
$\xi^2=\sigma$, whence $\xi^q=\sigma/\xi=\xi$ and so $\xi\in\F_q$,
contrary to the irreducibility of $F$.
\end{proof}

Ahmadi's generalization of Carlitz's result follows from Theorem~\ref{thm:Meyn} through an application of M\"obius inversion.
For the reader's convenience we recall a form of M\"obius inversion which is only slightly more general
than the classical one, see~\cite[Proposition~5.2]{Knopfmacher}.
Given a completely multiplicative function $\chi:\N\to\C$
(that is, a homomorphism of the multiplicative monoid $\N$ of the positive integers into the multiplicative monoid
of the complex numbers),
two functions $f,g:\N\to\C$ satisfy
\[
f(n)=\sum_{d\mid n}\chi(d)g(n/d)
\]
for all $n\in\N$ if, and only if, they satisfy
\[
g(n)=\sum_{d\mid n}\mu(d)\chi(d)f(n/d)
\]
for all $n\in\N$, where $\mu$ is the M\"obius function.
This allows one to invert relations of the form
$f(n)=\sum_{d\mid n,\ \text{$d$ odd}}g(n/d)$,
for example, by taking $\chi(d)=0$ for $d$ even and $\chi(d)=1$ for $d$ odd.
(This special case is~\cite[Theorem~2.7.2]{Jungnickel}.)

\begin{theorem}[Theorem~2 in~\cite{Ahmadi:Carlitz}]\label{thm:Ahmadi}
Let $g,h\in\F_q[x]$ be coprime polynomials with
$\max(\deg g,\deg h)=2$.
Then the number of monic irreducible polynomials $f\in\F_q[x]$ of degree $n>1$
such that
$\bigl(h(x)\bigr)^n\cdot f\bigl(g(x)/h(x)\bigr)$
is irreducible equals
\[
\begin{cases}
0
 &\text{if $q$ is even and $g'=h'=0$,}
\\
\frac{1}{2n}(q^n-1)
 &\text{if $q$ is odd and $n$ is a power of $2$,}
\\
\frac{1}{2n}\sum_{\substack{d\mid n\\\text{$d$ odd}}}
\mu(d)q^{n/d}
 &\text{otherwise.}
\end{cases}
\]
\end{theorem}

\begin{proof}
According to the discussion which precedes Theorem~\ref{prop:canonical-deg2},
the count of irreducible polynomials of the form described does not change upon pre- and post-composing $g(x)/h(x)$ with
affine maps or the inversion map.
Theorem~\ref{prop:canonical-deg2} then describes the resulting convenient forms to which $g(x)/h(x)$ can be brought.
In particular, the proof of Theorem~\ref{prop:canonical-deg2} shows that $g(x)/h(x)$ can be brought to the form $x^2$
exactly when $q$ is even and $g'=h'=0$.
This case does not contribute any irreducible polynomials of the desired form,
as $f(x^2)$ cannot be irreducible.
In all other cases $g(x)/h(x)$ can be brought to the form
$x+\sigma x^{-1}$ for some $\sigma\in K^{\ast}$.

Let $\SRIM_{\sigma}(2n,q)$ be the number of irreducible monic
polynomials of degree $2n$ in $\mathcal{I}_\sigma$.
Taking degrees in Theorem~\ref{thm:Meyn} we find
\[
q^n-\varepsilon^n=
\sum_{d\mid n,\ \text{$d$ odd}}2n/d\cdot\SRIM_{\sigma}(2n/d,q),
\]
where $\varepsilon=0$ for $q$ even, and
$\varepsilon=\pm 1\in\Z$ according as
$\sigma^{(q-1)/2}=\pm 1\in\F_q$ for $q$ odd.
M\"obius inversion as described above turns this equation into
\[
2n\cdot\SRIM_{\sigma}(2n,q)=\sum_{d\mid n,\ \text{$d$ odd}}\mu(d)(q^{n/d}-\varepsilon^{n/d}).
\]
Because the sum
$\sum_{d\mid n,\ \text{$d$ odd}}\mu(d)\varepsilon^{n/d}
=\varepsilon^{n}\sum_{d\mid n,\ \text{$d$ odd}}\mu(d)$
vanishes unless $q$ is odd and $n$ is a power of $2$, and in this case equals $\varepsilon^n$, we conclude
\begin{equation}\label{eq:SRIM_sigma}
\SRIM_{\sigma}(2n,q)=
\begin{cases}
\frac{1}{2n}(q^n-\varepsilon^n)
 &\text{if $q$ is odd and $n$ is a power of $2$,}
\\
\frac{1}{2n}\sum_{\substack{d\mid n\\\text{$d$ odd}}}
\mu(d)q^{n/d}
 &\text{otherwise.}
\end{cases}
\end{equation}
Because of our assumption $n>1$ we have $\varepsilon^n=1$ in Equation~\eqref{eq:SRIM_sigma},
and our proof is complete.
\end{proof}

The hypothesis $n>1$ in our Theorem~\ref{thm:Ahmadi}, as well as in~\cite{Ahmadi:Carlitz},
which was not required in Carlitz's Equation~\eqref{eq:Carlitz},
was needed to ensure that
$f_R(x)=\bigl(h(x)\bigr)^{\deg f}\cdot f\bigl(g(x)/h(x)\bigr)$
has degree equal to $2\deg f$.
In the excluded case $f(x)=x-\alpha$, for some $\alpha\in\F_q$, that conclusion fails exactly when
$g_2=\alpha h_2$,
where $g(x)=g_2x^2+g_1x+g_0$
and $h(x)=h_2x^2+h_1x+h_0$.
For completeness we now count the irreducible quadratic polynomials
which arise from polynomials $x-\alpha$ through a given quadratic transformation,
that is, those of the form $g(x)-\alpha h(x)$ for some $\alpha\in\F_q$.
To obtain a simpler statement we exclude the case of even characteristic
where both $g(x)$ and $h(x)$ are polynomials in $x^2$,
whence no irreducible polynomial can arise anyway.

\begin{theorem}\label{thm:n=1}
Let $g,h\in\F_q[x]$ be coprime polynomials with
$\max(\deg g,\deg h)=2$, and if $q$ is even assume that $g'$ and $h'$ are not both zero.
Then the number of monic irreducible quadratic polynomials which are $\F_q$-linear combinations of
$g(x)$ and $h(x)$ equals $q/2$ if $q$ is even,
and it equals
$(q-1)/2$ or $(q+1)/2$ if $q$ is odd,
according to whether the polynomial $g'h-gh'$ has its roots in $\F_q$, or not.
\end{theorem}

We omit the proof, which is similar to that of the general case $n>1$ in Theorem~\ref{thm:Ahmadi},
except that the reduction of $g(x)/h(x)$ to the special form $x+\sigma/x$ done in the proof of Theorem~\ref{prop:canonical-deg2}
needs to be adapted in order to avoid applying post-composition with the inversion map, where $\deg f_R=2\deg f$ may fail.

The following immediate corollary of Theorems~\ref{thm:Ahmadi} and~\ref{thm:n=1} states the special case of our count of irreducible polynomials
where they are closest to the traditional definition of self-reciprocal polynomials,
namely invariant under the involutive transformation considered in
Lemma~\ref{lemma:invariant_quadratic}.

\begin{cor}
Let $\sigma\in\F_q^{\ast}$.
The number of monic irreducible polynomials $g\in \F_q[x]$ of degree $2n$
which satisfy
$x^{2n}\cdot g(\sigma/x)=\sigma^ng(x)$ equals
\[
\frac{1}{2n}\biggl(-\delta+\sum_{\substack{d\mid n\\\text{$d$ odd}}}\mu(d)q^{n/d}\biggr),
\]
where
\[
\delta=
\begin{cases}
1 &\text{if $q$ is odd and $n>1$ is a power of $2$,}
\\
1 &\text{if $q$ is odd, $n=1$, and $\sigma$ is a square in $\F_q$,}
\\
-1 &\text{if $q$ is odd, $n=1$, and $\sigma$ is not a square in $\F_q$,}
\\
0 &\text{otherwise.}
\end{cases}
\]
\end{cor}

\section{Explicit characterization of the polynomials obtained through a quadratic transformation}\label{sec:explicit}

Meyn's proof in~\cite{Meyn} of Carlitz's counting formula for {\em srim} polynomials relies
on viewing them as irreducible factors of polynomials of the form $x^{q^n+1}-1$,
as in the special case $\sigma=1$ of our Theorem~\ref{thm:Meyn}, which is~\cite[Theorem~1]{Meyn}.
It is actually possible to obtain a similar characterization for an arbitrary quadratic transformation, as follows.

\begin{theorem}\label{thm:Meyn_generalized}
Let $g(x)=g_2x^2+g_1x+g_0$
and $h(x)=h_2x^2+h_1x+h_0$
be coprime polynomials over the field $\F_q$ of $q$ elements, with
$\max(\deg g,\deg h)=2$.
For any nonzero polynomial $f\in\F_q[x]$, further satisfying $f(g_2/h_2)\neq 0$ in case $h_2\neq 0$, we set
$f_R(x)=h(x)^{\deg f}\cdot f\bigl(g(x)/h(x)\bigr)$.
If $q$ is even, assume in addition that $g_1$ and $h_1$ are not both zero.

Then every irreducible polynomial of the form $f_R(x)$ for some $f(x)$, and of degree $2n/d$ with $d$ odd, is a factor of the polynomial
\[
H(x)=H_{R,q^n}(x)=ax^{q^n+1}-b(x^{q^n}+x)+c,
\]
where
\[
a=g_2h_1-g_1h_2,
\quad
b=g_0h_2-g_2h_0,
\quad
c=g_1h_0-g_0h_1.
\]

Furthermore, every irreducible factor of $H(x)$ of degree higher than one,
and different from $ax^2-2bx+c$ in case that is irreducible,
has (up to a scalar factor) the form $f_R(x)$ for some $f\in\F_q[x]$,
and its degree divides $2n$ but not $n$.
\end{theorem}

Some comments are in order on the statement of Theorem~\ref{thm:Meyn_generalized}.
The coprimality condition imposed on $g(x)$ and $h(x)$ in Theorem~\ref{thm:Meyn_generalized},
and the assumption $\max(\deg g,\deg h)=2$, are together equivalent to $b^2-ac\neq 0$.
This can be seen by computing the resultant of $g(x)$ and $h(x)$, or rather their quadratic homogenized versions.
In turn, for a polynomial $H(x)$ of the form given in Theorem~\ref{thm:Meyn_generalized},
the condition $b^2-ac\neq 0$ is equivalent to $H(x)$ having only simple roots in a splitting field, as
$(ax-b)H'(x)-aH(x)=b^2-ac$.
(Strictly speaking, this is true unless $a=b=0\neq c$, whence $H(x)$ is a nonzero constant,
but that case cannot occur under the hypotheses of Theorem~\ref{thm:Meyn_generalized}.)
In conclusion, the polynomial $H(x)$ of Theorem~\ref{thm:Meyn_generalized} has distinct roots in a splitting field,
and hence its irreducible factors over $\F_q$ are all distinct.

In the excluded case in Theorem~\ref{thm:Meyn_generalized} where $q$ is even and $g_1=h_1=0$,
the polynomial $f_R(x)$ belongs to $\F_q[x^2]$, hence is a square in $\F_q[x]$, and cannot be irreducible.

Our proof of Theorem~\ref{thm:Meyn_generalized} involves applying the quadratic transformation to reducible polynomials as well.
A problem arises, which we discussed near the beginning of Section~\ref{sec:quadratic},
and also affects the omitted proof of Theorem~\ref{thm:n=1},
of the degree of the transformed polynomial $f_R$ possibly being less than twice the degree of $f$.
As discussed there, this drop in degree occurs precisely when $(g/h)(\infty)$ is a root of $f$.
Hence we have avoided defining $f_R$ for such $f$ in Theorem~\ref{thm:Meyn_generalized}
by assuming that $f(g_2/h_2)\neq 0$ in case $h_2\neq 0$.

The proof of Theorem~\ref{thm:Meyn_generalized} requires a generalization of Lemma~\ref{lemma:invariant_quadratic},
which holds over an arbitrary field $K$,
where the involutory substitution $x\mapsto\sigma/x$ is replaced with an arbitrary involution in the M\"obius group over $K$.
Any such involution has the form
$x\mapsto (bx-c)/(ax-b)$,
for some $a,b,c\in K$ with $b^2-ac\neq 0$, and $a,c$ not both zero in case $q$ is even.
Note that its fixed points in $K$, if any, are the roots of $ax^2-2bx+c$.
Our Lemma~\ref{lemma:invariant_generalized} below roughly says that
the polynomials $F \in K[x]$ of even degree which are `invariant', in an appropriate sense, under the involution $x\mapsto (bx-c)/(ax-b)$,
are exactly those
which are obtained through a certain quadratic transformation, associated to $R(x)=g(x)/h(x)$ in the usual way.
There is some freedom as to the choice of $R(x)$ in the formulation, all choices being related by post-composition with
M\"obius transformations.

\begin{lemma}\label{lemma:invariant_generalized}
Let $K$ be any field, and let $a,b,c\in K$ with $b^2-ac\neq 0$.
Let $(g_0,g_1,g_2)$ and $(h_0,h_1,h_2)$ be $K$-linearly independent triples of elements of $K$ such that
$ag_0+bg_1+cg_2=0$ and $ah_0+bh_1+ch_2=0$.
If $K$ has characteristic two, assume in addition that $a$ and $c$ are not both zero.
Consider the two polynomials $g(x)=g_2x^2+g_1x+g_0$
and $h(x)=h_2x^2+h_1x+h_0$ in $K[x]$.

Then a polynomial $F \in K[x]$ of degree $2n$ satisfies
\begin{equation}\label{eq:invariant}
(ax-b)^{2n}\cdot
F\left(\frac{bx-c}{ax-b}\right)
=(b^2-ac)^n\cdot F(x)
\end{equation}
if, and only if, $F(x)=h(x)^n\cdot f\bigl(g(x)/h(x)\bigr)$ for some polynomial $f\in K[x]$.
\end{lemma}

Note that, despite the shift of focus from $g(x)$ and $h(x)$ to the triple $(a,b,c)$ in Lemma~\ref{lemma:invariant_generalized},
that triple is necessarily proportional to the triple $(a,b,c)$ constructed from $g(x)$ and $h(x)$ in Theorem~\ref{thm:Meyn_generalized}.
In particular, the comment we made on the hypothesis $b^2-ac\neq 0$ after the statement of Theorem~\ref{thm:Meyn_generalized} still applies,
and hence the rational expression $g(x)/h(x)$ produced in Lemma~\ref{lemma:invariant_generalized} is indeed quadratic
(that is to say, $\max(\deg g,\deg h)=2$).

At one point in the following proof, as well as in Section~\ref{sec:variations}, we will need to consider rational expressions of arbitrary degree
(over a field $K$).
Recall that the degree of a nonzero rational expression $u(x)=g(x)/h(x)\in K(x)$, where $g(x),h(x)\in K[x]$ are coprime polynomials,
is defined as $\deg(u)=\max(\deg g,\deg h)$.
This terminology is justified by the fact that, assuming $u(x)\not\in K$,
the degree of $u(x)$ equals the degree of the field extension $K(x)$ over $K(u)$.
In fact, the minimal polynomial of $x$ over $K(u)$ is a scalar multiple of $g(y)-uh(y)$.
These facts are often assigned as exercises in standard algebra textbooks, but a proof is explicitly given
in~\cite[Chapter~5, Proposition~2.1]{Cohn:Algebra3}

\begin{proof}[Proof of Lemma~\ref{lemma:invariant_generalized}]
We consider separately the special case where $a=0$,
which is equivalent to some linear combination of $g(x)$ and $h(x)$ being a nonzero constant.
In this case Equation~\eqref{eq:invariant} becomes
$F(c/b-x)=F(x)$, which can be shown to be equivalent to $F(x)$ being a polynomial in $bx^2-cx$.
Because each of $g(x)$ and $h(x)$ equals a scalar multiple of $bx^2-cx$ plus a constant,
the latter is equivalent to $F(x)$ being a rational function of $g(x)/h(x)$.
Deducing that $F(x)$ has the form described in Lemma~\ref{lemma:invariant_generalized}
for some polynomial $f$ (rather than just a rational function $f$)
can be done in the same way as in the case $a\neq 0$, which will be explained in the final part of this proof.

Now we may assume $a\neq 0$.
The polynomial
\[
(y-x)\left(y-\frac{bx-c}{ax-b}\right)=y^2-\frac{ax^2-c}{ax-b}y+\frac{bx^2-cx}{ax-b}
=y^2-zy+\frac{bz-c}{a}
\]
has coefficients in the subfield
$L=K(z)$ of $K(x)$, where
$z=(ax^2-c)/(ax-b)$.
It is irreducible over $L$
because its two roots in $K(x)$ are interchanged by the automorphism of $K(x)$ given by pre-composition
(that is, substitution) with the involution
$x\mapsto(bx-c)/(ax-b)$.
The linear conditions imposed in the hypotheses on the coefficients of $g(x)$ and $h(x)$
show that each of those two polynomials is a linear combination of the numerator and the denominator of
$(ax^2-c)/(ax-b)$.
Hence $g(x)/h(x)$ can be obtained from $(ax^2-c)/(ax-b)$ by post-composing it with a suitable M\"obius transformation.
In other words, $g(x)/h(x)$ can be obtained from $z$ by an application of a suitable M\"obius transformation,
and hence $L=K\bigl(g(x)/h(x)\bigr)$.

If
$F(x)/h(x)^n=f(g(x)/h(x))$
for some $f\in K[x]$, then the left-hand side must be invariant under substitution with
$x\mapsto(bx-c)/(ax-b)$,
and a calculation shows that this condition is equivalent to Equation~\eqref{eq:invariant}.
Conversely, if $F(x)/h(x)^n$
is invariant under the substitution $x\mapsto(bx-c)/(ax-b)$, then because
$L=K\bigl(g(x)/h(x)\bigr)$ we have
$F(x)/h(x)^n=f\bigl(g(x)/h(x)\bigr)$
for some rational expression $f\in K(x)$, necessarily of degree $n$.
We only need to show that $f$ is actually a polynomial.
If it were not, then it would have a pole at some
$\eta\in\overline{K}$, the algebraic closure of $K$.
But then $f\bigl(g(x)/h(x)\bigr)$
would have a pole at each root $\xi\in\overline{K}$ of
the polynomial $g(x)-\eta h(x)$.
Now $F(x)/h(x)^n=f\bigl(g(x)/h(x)\bigr)$
cannot have any pole except at any root $\zeta$ of $h$, but clearly $g(\zeta)-\eta h(\zeta) \neq 0$, hence we get the desired contradiction
and we are bound to conclude that $f\in K[x]$.
\end{proof}

In a similar way as for the condition of being self-reciprocal, which it generalizes,
Equation~\eqref{eq:invariant} in Lemma~\ref{lemma:invariant_generalized} can be checked in terms of the roots of $F$ in a splitting field, as follows.

\begin{lemma}\label{lemma:checking}
Assume that $F(x)$ in Lemma~\ref{lemma:invariant_generalized} is coprime with $ax^2-2bx+c$.
Then Equation~\eqref{eq:invariant}
is equivalent to $(b\xi-c)/(a\xi-b)$ being a root of $F(x)$ in a splitting field along with each root $\xi$, and with the same multiplicity.
\end{lemma}

\begin{proof}
To see this, assuming $F(x)$ monic as we may, and writing it as
$F(x)=\prod_{i=1}^{2n}(x-\xi_i)$ over a splitting field, we have
\begin{align*}
(ax-b)^{2n}\cdot
F\left(\frac{bx-c}{ax-b}\right)
&=
\prod_{i=1}^{2n}\bigl((bx-c)-\xi_i(ax-b)\bigr)
\\&=
\prod_{i=1}^{2n}(a\xi_i-b)
\cdot
\prod_{i=1}^{2n}\left(x-\frac{b\xi-c}{a\xi-b}\right).
\end{align*}
Hence if Equation~\eqref{eq:invariant} holds, then for every root $\xi$ of $F$ in a splitting field,
$(b\xi-c)/(a\xi-b)$ is a root as well, and with the same multiplicity.
Conversely, if the latter holds then Equation~\eqref{eq:invariant} holds up to a scalar factor.
To check that that factor is one we use our assumption that  $F(x)$ is coprime with $ax^2-2bx+c$,
and hence we may assume that $\xi_{n+i}=(b\xi_i-c)/(a\xi_i-b)$
for $n<i\le 2n$.
Consequently, we have
$a\xi_{n+i}-b=(b^2-ac)/(a\xi_i-b)$,
from which we conclude that
$\prod_{i=1}^{2n}(a\xi_i-b)=(b^2-ac)^n$
as desired.
\end{proof}

We are now ready to prove Theorem~\ref{thm:Meyn_generalized}.

\begin{proof}[Proof of Theorem~\ref{thm:Meyn_generalized}]
If $f(x)=\sum_{k=0}^n a_kx^k$, with $a_n\neq 0$, then the coefficient of $x^{2n}$ in $f_R(x)$
equals $\sum_{k=0}^n a_kg_2^kh_2^{n-k}$, which equals $h_2^n\cdot f(g_2/h_2)$ if $h_2\neq 0$,
and $a_ng_2^n$ otherwise.
Consequently, the polynomials $f$ under consideration satisfy
$\deg f_R=2\deg f$.
It readily follows that the {\em quadratic transformation}
$f(x)\mapsto f_R(x)$
preserves multiplication of such polynomials, in the sense that
$(f_1\cdot f_2)_R(x)=(f_1)_R(x)\cdot(f_2)_R(x)$.
In particular, $f_R(x)$ can possibly be irreducible only if $f(x)$ is.

Now suppose that $f_R(x)$ is irreducible, of degree $2n/d$ with $d$ odd, for some permitted $f(x)$.
Because $f(x)$ is irreducible of degree a divisor of $n$, it divides $x^{q^n}-x$.
Consequently, $f_R(x)$ divides
\[
(x^{q^n}-x)_R
=h(x)^{q^n}\cdot\left(\frac{g(x)^{q^n}}{h(x)^{q^n}}-\frac{g(x)}{h(x)}\right)
=\frac{g(x)^{q^n}h(x)-g(x)h(x)^{q^n}}{h(x)}.
\]
Now we have
\begin{align*}
g(x)^{q^n}h(x)-g(x)h(x)^{q^n}
&=
(g_2h_1-g_1h_2)x^{2q^n+1}
+(g_2h_0-g_0h_2)x^{2q^n}
\\&\quad+
(g_1h_2-g_2h_1)x^{q^n+2}
+(g_1h_0-g_0h_1)x^{q^n}
\\&\quad+
(g_0h_2-g_2h_0)x^{2}
+(g_0h_1-g_1h_0)x
\\&=
(x^{q^n}-x)\cdot H_{R,q^n}(x),
\end{align*}
where
\[
H_{R,q^n}(x)=ax^{q^n+1}-b(x^{q^n}+x)+c,
\]
having set
\[
a=g_2h_1-g_1h_2,
\quad
b=g_0h_2-g_2h_0,
\quad
c=g_1h_0-g_0h_1.
\]
Because $f_R(x)$ is irreducible of degree not dividing $n$, it cannot divide $x^{q^n}-x$, and hence it must divide $H(x)$.

Conversely, let $F(x)$ be any irreducible factor of $H(x)$.
Then
\[
x^{q^n}\equiv
\frac{bx-c}{ax-b}
\pmod{F(x)},
\]
and hence
\[
x^{q^{2n}}
\equiv
\left(\frac{bx-c}{ax-b}\right)^{q^n}
=
\frac{bx^{q^n}-c}{ax^{q^n}-b}
\equiv
\frac{b(bx-c)-c(ax-b)}{a(bx-c)-b(ax-b)}
=x
\pmod{F(x)}.
\]
Hence $F(x)$ divides $x^{q^{2n}}-x$ but not $x^{q^n}-x$.
(In particular, $\F_{q^{2n}}$ contains a splitting field for $F(x)$.)
Consequently, $F(x)$ has degree a divisor of $2n$ which is not a divisor of $n$.

We note in passing that the argument employed in the previous paragraph
has a natural extension to a M\"obius transformation
$x\mapsto(\gamma x+\delta)/(\alpha x+\beta)$
of higher order.
In fact, it provides information on the order, and consequently on the factorization in $\F_q[x]$,
of polynomials of the form $\alpha x^{q^n+1}+\beta x^{q^n}-\gamma x-\delta$,
see~\cite[Proposition~2.3]{Mat:nonsing-der-2}.
Such factorizations have been further investigated in~\cite{StiTop}.

It remains to prove that $F(x)=f_R(x)$ for some $f\in\F_q[x]$, which we do by an application of Lemma~\ref{lemma:invariant_generalized}.
The triple $(a,b,c)$ defined above
is the standard cross product of $(h_0,h_1,h_2)$ and $(g_0,g_1,g_2)$,
and hence is orthogonal to both of them with respect to the standard scalar product in $\F_q^3$.
This ensures that the conditions stated in the first paragraph of Lemma~\ref{lemma:invariant_generalized} are met.
(In case $K$ has characteristic two, our assumption that $g_1$ and $h_1$ are not both zero
implies that $a$ and $c$ are not both zero.)
Thus, we only need to check that Equation~\eqref{eq:invariant} is satisfied,
which we may do in terms of the roots of $F(x)$ in a splitting field, according to Lemma~\ref{lemma:checking}.
We have seen that $F(x)$ has all its roots in $\F_{q^{2n}}$.
If $\xi$ is any of them,
then $\xi^{q^n}=(b\xi-c)/(a\xi-b)$ is also a root, and clearly both are simple roots.
According to Lemma~\ref{lemma:checking} we conclude that Equation~\eqref{eq:invariant} is satisfied, as desired.
\end{proof}

In summary, Theorem~\ref{thm:Meyn_generalized} tells us that, under its hypotheses and up to a scalar factor,
the product of all irreducible polynomials of the form $f_R(x)$
of degree a divisor of $2n$ which does not divide $n$ equals
\[
\frac{ax^{q^n+1}-b(x^{q^n}+x)+c}{(ax^2-2bx+c,x^{q^n}-x)},
\]
where $a,b,c$ are obtained from $R(x)=g(x)/h(x)$ as described there.
Compare with Theorem~\ref{thm:Meyn}, where $(a,b,c)=(1,0,-\sigma)$.
The degree of this product polynomial equals $q^n-\varepsilon^n$,
where $\varepsilon=0$ for $q$ even, and
$\varepsilon=\pm 1\in\Z$ for $q$ odd according as to whether
$b^2-ac$ is a square or a nonsquare in $\F_q$.
Theorem~\ref{thm:Ahmadi} would follow again by an application of M\"obius inversion.

Irreducible factors of polynomials of the form $H(x)=ax^{q^n+1}-b(x^{q^n}+x)+c$ as in Theorem~\ref{thm:Meyn_generalized}
were already considered in~\cite{StiTop}.
In essence, they were characterized in~\cite[Theorem~4.2]{StiTop} as those irreducible polynomials which are invariant under a certain transformation,
expressed by our Equation~\eqref{eq:invariant}.
While the remainder of~\cite{StiTop} focuses on asymptotic counting results,
our Theorem~\ref{thm:Meyn_generalized} provides an explicit construction for those irreducible factors
as resulting from the application of the appropriate quadratic transformation.

\section{Variations on Lemma~\ref{lemma:invariant_quadratic}}\label{sec:variations}

In Section~\ref{sec:counting} we have chosen to give what we feel is the simplest and most direct proof of Lemma~\ref{lemma:invariant_quadratic},
but several other lines of proof are possible, which we outline here.
We can clearly restrict ourselves to discussing the only nontrivial implication, namely, the existence of $f$ given $F$.

One possibility is a reduction to the well-known special case $\sigma=1$ of self-reciprocal polynomials,
which can be done by extending the field $K$ to one containing a square root $\rho$ of $\sigma$.
In fact, upon substituting $x$ with $\rho x$, the condition
$x^{2n}\cdot F(\sigma/x)=\sigma^nF(x)$
becomes
$x^{2n}\cdot \tilde F(1/x)=\tilde F(x)$
in terms of $\tilde F(x)=F(\rho x)$.
This means that $\tilde F(x)$ is self-reciprocal.
An appeal to that special case followed by the inverse substitution
produces the desired polynomial $\tilde f\in K(\rho)[x]$, and it only remains to check that $f$ actually has coefficients in $K$.
We omit the details.

Another proof uses a classical argument of field theory and relies on the fact that $K(x+\sigma/x)$ is the fixed subfield of
the automorphism of $K(x)$ given by $x\mapsto\sigma/x$.
There is no need to spell out this proof either, as it is a special case of our proof of Lemma~\ref{lemma:invariant_generalized} above.
This argument easily transfers to other situations, as in the
proof of Lemma~\ref{lemma:invariant_cubic} below.

The proofs of Lemma~\ref{lemma:invariant_quadratic} which we have described so far are not constructive.
A simple proof by induction on $n$ (as in~\cite[Lemma~2.75]{Jungnickel} for the special case $\sigma=1$) produces an algorithm for
recovering $f$ from $F$.
However, one can actually write
an explicit formula for $f$ in terms of $F$ using {\em Dickson polynomials}.
Recall that the {\em Dickson polynomial of the first kind} of degree $n$, for $n\ge 0$, is
\[
D_n(x,a)=\sum_{i=0}^{\lfloor n/2\rfloor}
\frac{n}{n-i}\binom{n-i}{i}(-a)^ix^{n-2i},
\]
see~\cite{LMT:Dickson} or~\cite{LN}.
The fundamental property of those Dickson polynomials, which can also be used to define them, is the functional equation
$D_n(x+a/x)=x^n+(a/x)^n$.
Now, in the setting of Lemma~\ref{lemma:invariant_quadratic},
if $F\in K[x]$ of degree $2n$
satisfies $x^{2n}\cdot F(\sigma/x)=\sigma^nF(x)$,
then $F(x)/x^n=b_n+\sum_{k=1}^nb_{n+k}(x^k+\sigma/x^k)$,
and hence
$F(x)=x^n\cdot f(x+\sigma/x)$,
where
$f(y)=b_n+\sum_{k=1}^n b_{n+k}D_k(y,\sigma)$.
Straightforward manipulation then leads to a formula for the coefficient of $y^j$ in $f(y)$
in terms of the coefficients of $F(x)$.
To keep that simple assume that $K$ has characteristic different from two, allowing us
to rewrite the central coefficient $b_n$ of $F(x)$ as $2b_n$, whence
$F(x)/x^n=\sum_{k=0}^nb_{n+k}(x^k+\sigma/x^k)$.
The coefficient of $y^j$ in $f(y)$ then equals
\[
\sum_{i=0}^{\lfloor(n-j)/2\rfloor}
\frac{2i+j}{i+j}\binom{i+j}{i}(-\sigma)^ib_{n+2i+j}.
\]

The definition of self-reciprocal polynomials in terms of appropriate invariance under the involutory substitution $x\mapsto 1/x$
prompts a natural
generalization of self-reciprocal polynomials, namely, polynomials which are invariant
under pre-composition with a M\"obius transformation of order $r$.
Such a generalization has been considered to some extent in~\cite{StiTop} and some of the references therein,
but here we focus on natural analogues of Lemma~\ref{lemma:invariant_quadratic}.

We may work over an arbitrary field $K$.
In case $K$ has positive characteristic $p$, a fundamental distinction is whether $p$ divides the order $r$ of the M\"obius transformation, or not.
We only mention an example of the former case before passing to the latter case, which is far more interesting.
Any M\"obius transformation of order $p$ is conjugate to the translation $x\mapsto x+1$.
One easily finds that any polynomial satisfying $F(x+1)=F(x)$ has the form $F(x)=f(x^p-x)$.

Under the assumption that the characteristic of $K$ does not divide $r$,
it is known from~\cite{Beauville:PGL},
that all subgroups of $\PGL(2,K)$ of order $r$ are conjugate for
$r>2$, while the conjugacy classes of subgroups (or elements) of order two are
in a natural correspondence with the elements of
$K^\ast/(K^\ast)^2$.
Lemma~\ref{lemma:invariant_quadratic} dealt with the latter case.
Note that elements of a given order $r>2$ need not be conjugate in $\PGL(2,K)$, but because the subgroups they generate are conjugate
there is essentially one higher analogue of Lemma~\ref{lemma:invariant_quadratic}
for every $r>2$, depending on a choice of an element of order $r$ in $\PGL(2,K)$.
We exemplify such results with the special cases $r=3,4$.
As representatives of elements of order 3 and 4 in $\PGL(2,K)$ we may take those represented by the matrices
$\left[\begin{smallmatrix}
0&1\\-1&1
\end{smallmatrix}\right]$,
and
$\left[\begin{smallmatrix}
0&1\\-2&2
\end{smallmatrix}\right]$,
for $q$ odd in the latter case.
This means considering the M\"obius transformations $x\mapsto 1/(1-x)$ and $x\mapsto 1/(2-2x)$,
which we do in our concluding results.

\begin{lemma}\label{lemma:invariant_cubic}
Let $K$ be a field and let $F(x)\in K[x]$ be a polynomial of degree $3n$.
We have $(x-1)^{3n}\cdot F\bigl(1/(1-x)\bigr)=F(x)$ if, and only if,
\[
F(x)=x^n(x-1)^n\cdot f\left(\frac{x^3-3x+1}{x(x-1)}\right)
\]
for some $f\in K[x]$
of degree $n$.
\end{lemma}

\begin{proof}
Consider the automorphism of the field $K(x)$ given by the substitution
$x\mapsto 1/(1-x)$ of order three.
The monic polynomial which has its distinct composition powers as its
roots is
\[
(y-x)
\left(y-\frac{1}{1-x}\right)
\left(y-\frac{x-1}{x}\right)
=
y^3
-\frac{x^3-3x+1}{x(x-1)}y^2
+\frac{x^3-3x^2+1}{x(x-1)}y
+1.
\]
Because the sum of the coefficients of $y^2$ and $y$ equals $-3$,
all coefficients belong to the subfield
$L=K\bigl(\frac{x^3-3x+1}{x(x-1)}\bigr)$ of $K(x)$.
Because $|K(x):L|=3$ equals the order of the substitution
$x\mapsto 1/(1-x)$,
we have that $K(x)$ is a Galois extension of $L$ with Galois group generated by
that substitution.

If
\begin{equation}\label{eq:gf3}
\frac{F(x)}{x^n(x-1)^n}=f\left(\frac{x^3-3x+1}{x(x-1)}\right)
\end{equation}
for some $f\in K[x]$, then the left-hand side must be invariant under the substitution
$x\mapsto 1/(1-x)$,
and
$(x-1)^{3n}\cdot F\bigl(1/(1-x)\bigr)=F(x)$
follows after a short calculation.
Conversely, if the latter holds then Equation~\eqref{eq:gf3} holds
for some rational expression $f\in K(x)$, necessarily of degree $n$.
If $f$ were not a polynomial, then it would have a pole at some
$\eta\in\overline{K}$,
the algebraic closure of $K$.
But then the right-hand side of Equation~\eqref{eq:gf3}, and hence
the left-hand side as well, would have a pole at any root $\xi\in\overline{K}$ of
the polynomial
$(x^3-3x+1)-\eta x(x-1)$.
Because this polynomial cannot have $0$ or $1$ as roots, this is
impossible.
We conclude that $f\in K[x]$, as desired.
\end{proof}

Differently from the previous discussion, we have not excluded that $K$ may have characteristic three in Lemma~\ref{lemma:invariant_cubic},
but in that case the substitution $x\mapsto 1/(1-x)$ is conjugate to $x\mapsto x+1$, an easy case which we have briefly discussed earlier on.

\begin{lemma}
Let $K$ be a field of characteristic not two, and let $F(x)\in K[x]$ be a polynomial of degree $4n$.
Then
\[
(-1/4)^{n}\cdot (2-2x)^{4n}\cdot F\bigl(1/(2-2x)\bigr)=F(x)
\]
holds if, and only if,
\[
F(x)=x^n(x-1)^n(x-1/2)^n\cdot f\left(\frac{x^4-3x^2+2x-1/4}{x(x-1)(x-1/2)}\right)
\]
for some $f\in K[x]$
of degree $n$.
\end{lemma}

We omit the proof, which is entirely similar to that of Lemma~\ref{lemma:invariant_cubic}, but
just point out that the argument of $f$ in the above equation for $F(x)$ equals
\[
x
+\frac{1}{2-2x}
+\frac{1-x}{1-2x}
+\frac{2x-1}{2x},
\]
the sum of the iterates of $1/(2-2x)$.

\bibliography{References}
\end{document}